\documentclass[11pt]{article}

\title{Quotients of CI-groups are CI-groups}

\author{Edward Dobson \\
Department of Mathematics and Statistics \\
Mississippi State University \\
PO Drawer MA Mississippi State, MS 39762 \\
dobson@math.msstate.edu\\
\\
Joy Morris\\
Department of Mathematics and Computer Science\\
University of Lethbridge\\
Lethbridge, AB T1K 3M4 Canada\\
joy.morris@uleth.ca}

\usepackage{amssymb}

\usepackage{theorem}

\usepackage{cite}

\usepackage{pdfsync}

\newtheorem{thrm}{Theorem}

\newtheorem{lem}[thrm]{Lemma}

\theorembodyfont{\rm} \theoremstyle{definition}
\newtheorem{defin}[thrm]{Definition}

 \newcounter{case}

 \renewcommand{\thecase}{\arabic{case}}

\newcounter{subcase}

 \renewcommand{\thesubcase}{\alph{subcase}}

\usepackage{fullpage}

\def\tl{\triangleleft}

\def\fix{{\rm fix}}
\def\Aut{{\rm Aut}}

\def\Cay{{\rm Cay}}
\def\sym{{\mathcal S}}
\newenvironment{proof}{\noindent {\sc Proof}.}
                {\phantom{a} \hfill \framebox[2.2mm]{ } \bigskip}

\begin{document}

\pagestyle{plain}

\baselineskip = 1.3\normalbaselineskip

\maketitle

\begin{abstract}
We show that a quotient group of a CI-group with respect to (di)graphs is a CI-group with respect to (di)graphs.
\end{abstract}

In \cite{BabaiF1978, BabaiF1979}, Babai and Frankl provided strong constraints on which finite groups could be CI-groups with respect to graphs.  As a tool in this program, they proved \cite[Lemma 3.5]{BabaiF1978} that a quotient group $G/N$ of a CI-group $G$ with respect to graphs is a CI-group with respect to graphs provided that $N$ is characteristic in $G$.  They were not able to prove that a quotient group of a CI-group with respect to graphs is a CI-group with respect to graphs in the general case, and so introduced the notion of a {\it weak CI-group with respect to graphs} in order to treat quotient groups of CI-groups.  In some sense, the program that Babai and Frankl started was completed by Li \cite{Li1999} when he showed that all CI-groups are solvable. (Babai and Frankl mention in \cite{BabaiF1979} a sequel to their first paper that addressed showing all CI-groups with respect to graphs are solvable. This sequel never appeared.)  We will show that a quotient group of a CI-group with respect to (di)graphs is a CI-group with respect to (di)graphs.  This will allow for a simplification of the proofs of Babai and Frankl in \cite{BabaiF1978, BabaiF1979} (for example the notion of a weak CI-group with respect to graphs will no longer be needed), and consequently, as Li's proof in \cite{Li1999} was based on the earlier work of Babai and Frankl, a simplification of the proof that a CI-group with respect to graphs is solvable.  We begin with some basic definitions.

\begin{defin}
Let $G$ be a group and $S\subset G$. Define a {\bf Cayley digraph of
$G$}, denoted $\Cay(G,S)$, to be the digraph with $V(\Cay(G,S)) =
G$ and $E(\Cay(G,S)) = \{(g,gs):g\in G, s\in S\}$.  We call $S$ the {\bf connection set of $\Cay(G,S)$}.  If $S = S^{-1}$, then $\Cay(G,S)$ is a graph.
\end{defin}

Typically, definitions of Cayley (di)graphs assume $1_G \not\in S$ to avoid loops, but this assumption is rarely material to proofs, and will not be made here.

It is straightforward to show that $g_L:G\to G$ by $g_L(x) = gx$ is always an automorphism of $\Cay(G,S)$, and so $G_L = \{g_L:g\in G\}$ is a subgroup of $\Aut(\Cay(G,S))$, the automorphism group of $\Cay(G,S)$.  $G_L$ is the {\bf left regular representation of $G$}.

\begin{defin}
We say that a group $G$ is a {\bf CI-group with respect to (di)graphs} if  given $\Cay(G,S)$ and $\Cay(G,S')$, $S,S'\subset G$, then $\Cay(G,S)$ and $\Cay(G,S')$ are isomorphic if and only if $\alpha(S) = S'$ for some $\alpha\in \Aut(G)$.
\end{defin}

It is also straightforward to verify that $\alpha(\Cay(G,S))=\Cay(G, \alpha(S))$ is a Cayley (di)graph of $G$ for every $S\subset G$ and $\alpha\in\Aut(G)$.  Thus if one is testing whether or not two Cayley (di)graphs of a group $G$ are isomorphic, one must always check whether or not there is a group automorphism of $G$ that acts as an isomorphism.  A CI-group with respect to (di)graphs is then a group where the group automorphisms of $G$ are the only maps which need to be checked to determine isomorphism.

We now state some of the definitions from permutation group theory that will be required.

\begin{defin}
Let $G$ be a transitive group acting on a set $X$.  A subset $B\subseteq X$ is a {\bf block} of $G$ if whenever $g\in G$, then $g(B)\cap B \in \{\emptyset,B\}$.  If $B = \{x\}$ for some $x\in X$ or $B = X$, then $B$ is a {\bf trivial block}. Any other block is nontrivial, and if $G$ admits nontrivial blocks then $G$ is {\bf imprimitive}.  If $G$ is not imprimitive, we say that $G$ is {\bf primitive}.  Note that if $B$ is a block of $G$, then $g(B)$ is also a block of $B$ for every $g\in G$, and is called a {\bf conjugate block of $B$}.  The set of all blocks conjugate to $B$, denoted ${\cal B}$, is a partition of $X$, and ${\cal B}$ is called a {\bf $G$-invariant partition of $X$}.
\end{defin}

\begin{defin}
Let ${\cal B}$ be a $G$-invariant partition.  Define $\fix_G({\cal B}) = \{g\in G:g(B) = B{\rm\ for\ all\ }B\in{\cal B}\}$.  That is, $\fix_G({\cal B})$ is the group of permutations in $G$ that simultaneously fixes each block of ${\cal B}$ set-wise.  If ${\cal C}$ is also a $G$-invariant partition and for every $C\in{\cal C}$ we have that $C\subset B$ for some $B \in \cal B$, we write ${\cal C}\preceq{\cal B}$.  So ${\cal C}$ is a refinement of ${\cal B}$.
\end{defin}

Wreath products of both groups and graphs will be crucial.

\begin{defin}
Let $G$ be a permutation group acting on $X$ and $H$ a permutation group acting on $Y$.  Define the {\bf wreath product of $G$ and $H$}, denoted $G\wr H$, to be the set of all permutations $f$ of $X\times Y$ for which there exists $g \in G$, and for every $x \in X$ there exists $h_x \in H$, such that $f((x,y))=(g(x),h_x(y))$.
\end{defin}

We remark that many authors reverse the order of $G$ and $H$ in $G\wr H$, and/or refer to the wreath product of graphs (see Definition \ref{graphwreath} below) as the lexicographic product.

The following result is certainly known by many readers.  It and its proof are included here for completeness.

\begin{lem}\label{wreathblocksunique}
Let $G$ and $H$ be transitive groups and ${\cal B}$ the $(G\wr H)$-invariant partition formed by the orbits of $1_G\wr H$.  If ${\cal C}$ is a $(G\wr H)$-invariant partition, then either ${\cal B}\preceq{\cal C}$ or ${\cal C}\preceq{\cal B}$.  Consequently, ${\cal B}$ is the only $(G\wr H)$-invariant partition with blocks whose length is the degree of $H$.
\end{lem}

\begin{proof}
Let ${\cal C}$ be a $(G\wr H)$-invariant partition, and $B\in{\cal B}$.  Let $K$ be the point-wise stabilizer of every point {\it not} in $B$.  Then $K$ is transitive on $B$.  Now, either ${\cal B}\preceq{\cal C}$ or not.  If so, we are finished.  If not, then let $C\in{\cal C}$ such that $C\cap B\not = \emptyset$.  Then there exists at least one  
element of $B$ not in $C$, and so there exists $k\in K$ such that $k(C)\not = C$.  Then $k(C)\cap C = \emptyset$ so that $k$ fixes no point of $C$.  But $k$ fixes every point not in $B$, and so $C\subseteq B$ and ${\cal C}\preceq{\cal B}$.
\end{proof}

\begin{defin}\label{graphwreath}
Let $\Gamma_1$ and $\Gamma_2$ be digraphs.  The {\bf wreath product of $\Gamma_1$ and $\Gamma_2$}, denoted $\Gamma_1\wr\Gamma_2$ is the digraph with vertex set $V(\Gamma_1)\times V(\Gamma_2)$ and edge set

$$\{(u,v)(u,v'):u\in V(\Gamma_1){\rm\ and\ }vv'\in E(\Gamma_2)\}\cup\{(u,v)(u',v'):uu'\in E(\Gamma_1){\rm\ and\ }v,v'\in V(\Gamma_2)\}.$$
\end{defin}

The following result \cite[Theorem 5.7]{DobsonM2009} giving the automorphism group of vertex-transitive wreath product (di)graphs will be useful.  In the statement, for a (di)graph $\Gamma$, $\bar{\Gamma}$ denotes the complement of $\Gamma$.

\begin{thrm}\label{graphwreath}
For any finite vertex-transitive (di)graph $\Gamma \cong \Gamma_1 \wr \Gamma_2$, if
$\Aut(\Gamma) \neq \Aut(\Gamma_1) \wr \Aut(\Gamma_2)$ then there are some natural
numbers $r>1$ and $s>1$ and vertex-transitive (di)graphs $\Gamma_1'$ and $\Gamma_2'$ for which either
\begin{enumerate}
\item $\Gamma_1 \cong \Gamma_1' \wr K_r$, $\Gamma_2 \cong K_s \wr \Gamma_2'$ or
\item $\Gamma_1\cong\Gamma_1'\wr\bar{K}_r$ and $\Gamma_2\cong \bar{K}_s \wr \Gamma_2'$,
\end{enumerate}
and $\Aut(\Gamma)=\Aut(\Gamma_1')\wr (\sym_{rs} \wr \Aut(\Gamma_2'))$.
\end{thrm}

\begin{thrm}
Let $G$ be a CI-group with respect to (di)graphs and $H\tl G$.  Then $G/H$ is a CI-group with respect to (di)graphs.
\end{thrm}

\begin{proof}
Let $\ell = \vert H\vert$, and $\Cay(G/H,S_1)$ and $\Cay(G/H,S_2)$ be isomorphic.  If $\Cay(G/H,S_1)\not = \Gamma_1\wr K_\ell$ for some (di)graph $\Gamma_1$ and $\ell\ge 2$, then $\Cay(G/H,S_2)\not = \Gamma_2\wr K_\ell$ for any (di)graph $\Gamma_2$ and $\ell\ge 2$.  In this case, define $T_1 = \{gh:gH\in S_1,h\in H\}\cup(H - \{1_G\})$ and $T_2 = \{gh:gH\in S_2, h\in H\}\cup(H - \{1_G\})$.  Then $\Cay(G,T_1) = \Cay(G/H,S_1)\wr K_\ell$ and $\Cay(G,T_2) = \Cay(G/H,S_2)\wr K_\ell$ are isomorphic Cayley (di)graphs of $G$.  Additionally, by Theorem \ref{graphwreath}, we have that $\Aut(\Cay(G,T_1)) = \Aut(\Cay(G/H,S_1))\wr \sym_\ell$ and $\Aut(\Cay(G,T_2)) = \Aut(\Cay(G/H,S_2))\wr \sym_\ell$.  On the other hand, if $\Cay(G/H,S_1) = \Gamma_1\wr K_\ell$ for some $\Gamma_1$ and $\ell\ge 2$, then $\Cay(G/H,S_2) = \Gamma_2\wr K_\ell$ for some $\Gamma_2$.  In this case, define $T_1 = \{gh:gH\in S_1,h\in H\}$ and $T_2 = \{gh:gH\in S_2,h\in H\}$.  Then $\Cay(G,T_1) = \Cay(G/H,S_1)\wr \bar{K}_\ell$ and $\Cay(G,T_2) = \Cay(G/H,S_2)\wr\bar{K}_\ell$ are isomorphic Cayley digraphs of $G$.  As before, by Theorem \ref{graphwreath}, we have that $\Aut(G,T_1) = \Aut(\Cay(G/H,S_1))\wr \sym_\ell$ and $\Aut(\Cay(G,T_2)) = \Aut(\Cay(G/H,S_2))\wr \sym_\ell$.  In either case, $\Cay(G,T_1)$ and $\Cay(G,T_2)$ are isomorphic Cayley digraphs of $G$ such that $\Aut(\Cay(G,T_1)) = \Aut(\Cay(G/H,S_1))\wr\sym_\ell$ and $\Aut(\Cay(G,T_2)) = \Aut(\Cay(G/H,S_2))\wr\sym_\ell$.

%
%
%

As $G$ is a CI-group with respect to (di)graphs, there exists $\alpha\in\Aut(G)$ such that $\alpha(\Cay(G,T_1)) = \Cay(G,\alpha(T_1))=\Cay(G,T_2)$.
Since both $\Cay(G,T_1)$ and $\Cay(G,T_2)$ have the form $\Gamma_1'\wr \Gamma_2'$ where $\Gamma_2'$ has order $\ell$, Lemma \ref{wreathblocksunique} tells us that there is a unique $\Aut(\Cay(G,T_1))$-invariant partition with blocks of length $\ell$ in $\Cay(G,T_1)$, and a unique $\Aut(\Cay(G,T_2))$-invariant partition with blocks of length $\ell$ in $\Cay(G,T_2)$, and furthermore that in each case, these block systems are formed by the orbits of $1_{\Aut(\Cay(G/H,S_i))} \wr S_{\ell}$. By inspecting the connection sets of $\Cay(G,T_1)$ and $\Cay(G,T_2)$, it is clear that in both graphs these orbits are the cosets of $H$ in $G$.  Since $\alpha$ is an isomorphism from $\Cay(G,T_1)$ to $\Cay(G,T_2)$, it must take the unique invariant partition with blocks of length $\ell$ in $\Cay(G,T_1)$, to the unique invariant partition with blocks of length $\ell$ in $\Cay(G,T_2)$, and hence take any coset of $H$ to a coset of $H$.  Since $\alpha \in \Aut(G)$ it takes subgroups of $G$ to subgroups of $G$, so in particular, $\alpha(H)=H$.

Now $\alpha$ induces an automorphism $\bar{\alpha}$ of $G/H$ defined by $\bar{\alpha}(gH)=\alpha(g)H$.  Since $\alpha(H)=H$, this is well-defined.  We claim that $\bar{\alpha}(\Cay(G/H,S_1)) = \Cay(G/H,\bar{\alpha}(S_1))=\Cay(G/H,S_2)$, and so $G/H$ is a CI-group with respect to digraphs.
To see this, suppose that $gH \in S_1$.  Then $\bar{\alpha}(gH)=\alpha(g)H$, and by the definition of $T_1$, $gh \in T_1$ for every $h \in H$.  Since $\alpha(T_1)=T_2$, this means that $\alpha(gh)=\alpha(g)\alpha(h) \in T_2$ for every $h \in H$, and since $\alpha(H)= H$, this means $\alpha(g)h \in T_2$ for every $h \in H$.  By definition of $T_2$, this means that $\bar{\alpha}(gH)=\alpha(g)H \in S_2$.  Since $gH$ was an arbitrary element of $S_1$, this shows that $\bar{\alpha}(S_1)=S_2$, as claimed.
\end{proof}

\providecommand{\bysame}{\leavevmode\hbox to3em{\hrulefill}\thinspace}
\providecommand{\MR}{\relax\ifhmode\unskip\space\fi MR }
\providecommand{\MRhref}[2]{%
  \href{http://www.ams.org/mathscinet-getitem?mr=#1}{#2}
}
\providecommand{\href}[2]{#2}

\end{document}